\documentclass[12pt,twoside]{amsart}
\usepackage{geometry}
\usepackage{amssymb,amsmath,amsthm, amscd, enumerate, mathrsfs}
\usepackage{graphicx, hhline}
\usepackage[shortlabels]{enumitem}
\setlist[enumerate]{leftmargin=56pt,labelsep=
8pt,itemsep=4pt,label=\upshape{(\thethm.\arabic*)}}
\usepackage{mathtools}
\usepackage{color}
\usepackage{hyperref}
\usepackage{tikz-cd}
\usepackage{tikz}
\usepackage{tabularray}
\usepackage{float}
\usepackage{graphicx}

\title{A note on \texorpdfstring{$\mathbb{Q}$-Gorenstein}{Q-Gorenstein} surfaces}
\author{Nao Moriyama} 
\address{Department of 
Mathematics, Graduate School of Science, 
Kyoto University, Kyoto 606-8502, Japan}
\email{moriyama.nao.22s@st.kyoto-u.ac.jp}

\newtheorem{thm}{Theorem}[section]
\newtheorem{lem}[thm]{Lemma}

\newtheorem{ques}[thm]{Question}

\theoremstyle{definition}
\newtheorem{defn}[thm]{Definition}
\newtheorem{rem}[thm]{Remark}

\newtheorem*{ack}{Acknowledgments}

\makeatletter

\@addtoreset{equation}{section}
\makeatother
\begin{document}

\begin{abstract}
We construct a normal projective $\mathbb{Q}$-Gorenstein surface over an algebraically closed field whose canonical ring is not finitely generated. Moreover, we provide a counterexample to the minimal model program for $\mathbb{Q}$-Gorenstein surfaces, which was previously unknown in positive characteristic.
\end{abstract}

\maketitle
\tableofcontents

\section{Introduction}
This paper aims primarily to extend several examples from \cite{Sak87} to the case of positive characteristic.
Throughout this paper, we work over an algebraically closed field $k$. Our primary focus is the following question, to which we answer \textit{No}.
\begin{ques}[{cf. \cite[Section 4]{Fuj17}}] \label{ques: 1}
Let $X$ be a normal projective $\mathbb{Q}$-Gorenstein surface over $k$. Is the canonical ring $R(X)$ always finitely generated?
\end{ques}
In several important cases, the answer to Question \ref{ques: 1} is \textit{Yes}.
For instance, it is known to hold if one of the following conditions holds: (i) $X$ is Gorenstein (cf.~ Remark \ref{rem: f.g. for Gorenstein}),
(ii) $X$ is $\mathbb{Q}$-factorial,
(iii) $k = \overline{\mathbb{F}}_p$,
(iv) $X$ has only rational singularities,
(v) $X$ is log canonical,
and (vi) the Kodaira dimension $\kappa(X) < 2$ (cf.~ Lemma \ref{lem: kappa=1}). 
Cases (ii)-(v) follow from \cite{Fuj12} and \cite{Tan14}; for further details, see also \cite{FT12} and \cite{Fuj21-MR}. 
We note that normal surface over $\overline{\mathbb{F}}_p$ is always $\mathbb{Q}$-factorial (cf.~ \cite[Theorem 4.5]{Tan14}).

Contrary to these positive results, we show that the finite generation does not always hold:
\begin{thm}[{cf. Section \ref{sec: f.g.}}]
\label{thm: counterexample to finite generation}
Assume that $k \neq \overline{\mathbb{F}}_p$. Then, there exists a normal projective $\mathbb{Q}$-Gorenstein surface $X$ over $k$ such that $\kappa(X) = 2$ but the canonical ring $R(X)$ is not finitely generated.
\end{thm}

In addition to the finite generation, we discuss the Minimal Model Program (MMP).
It is known that the cone and contraction theorem holds for $\mathbb{Q}$-Gorenstein surfaces (see, for example, \cite[Theorems 5.5 and 5.6]{Fuj21-MR}).

\begin{ques} \label{ques: 2}
Does the MMP always work for normal projective $\mathbb{Q}$-Gorenstein surfaces over $k$?
\end{ques}
As with Question \ref{ques: 1}, the answer is known to be \textit{Yes} if $X$ satisfies any of the conditions (i)--(v) mentioned above. Case (i) follows from Lemma \ref{lem: MMP for Gorenstein} and other cases are similarly established by \cite{Fuj12}, \cite{Tan14}, and related works.

Our second main result provides a counterexample to the preservation of the $\mathbb{Q}$-Gorensteinness under the MMP:
\begin{thm}[cf. Section \ref{sec: MMP}] 
\label{thm: counterexample to MMP}
Assume that $k \neq \overline{\mathbb{F}}_p$. Then, there exists a normal projective $\mathbb{Q}$-Gorenstein surface $\tilde{Y}$ over $k$ and a contraction morphism $\pi \colon \tilde{Y} \to Y$ induced by a $K_{\tilde{Y}}$-negative extremal ray such that $Y$ is no longer $\mathbb{Q}$-Gorenstein.
\end{thm}

Examples similar to those in Theorems \ref{thm: counterexample to finite generation} and \ref{thm: counterexample to MMP} are known when $k=\mathbb{C}$ (cf.~ \cite{Sak87}). 
However, their constructions rely on analytic methods, such as the Grauert--Sakai contraction criterion.
In this paper, we provide a purely algebraic construction of these examples. This approach allows us to extend such counterexamples to positive characteristic (excluding $k = \overline{\mathbb{F}}_p$).

Here, we review the status of the finite generation and the MMP for normal projective surfaces over $k$.
\begin{itemize}
    \item Case $k = \overline{\mathbb{F}}_p$: The finite generation and the MMP always hold.
    \item Case $k \neq \overline{\mathbb{F}}_p$ (including $\mathrm{char}(k) = 0$): The situation depends on the types of singularities as follows:
    \begin{table}[ht]
\centering
\label{tab:comparison}
\begin{tabular}{|l|c c|}
\hline
\textbf{Condition} &  \multicolumn{2}{|c|}{Finite generation and MMP}\\
\hline
Gorenstein &
\quad \checkmark \quad & Lemma \ref{lem: MMP for Gorenstein} and Remark \ref{rem: f.g. for Gorenstein}\\ 
{$\mathbb{Q}$-factorial, log canonical} & \quad \checkmark \quad  & \cite{Fuj12} and \cite{Tan14}\\ 
$\mathbb{Q}$-Gorenstein & \quad $\times$ \quad & Theorems \ref{thm: counterexample to finite generation} and \ref{thm: counterexample to MMP} \\ 
\hline
\end{tabular}
\end{table}
\end{itemize}
It should be noted that Gorenstein, $\mathbb{Q}$-factorial, and log canonical properties are all strictly stronger than $\mathbb{Q}$-Gorenstein property. In the framework of the MMP, it is natural to consider a class of varieties where $K_X$ is $\mathbb{Q}$-Cartier and this property is preserved under contractions. The $\mathbb{Q}$-Gorenstein class is the largest such category that comes to mind; however, our results show that, unfortunately, the MMP does not necessarily work within this class.

For further results on the minimal model theory for log surfaces, we refer to \cite{Tan18}, \cite{Tan20}, \cite{Fuj21} or \cite{Fuj21-MR}.

Furthermore, in Section \ref{sec: criterion}, we discuss a criterion for the finite generation of $R(X)$ for minimal surfaces, which does not seem to have been explicitly noted in positive characteristic.

\begin{ack}
The author would like to thank Professor Osamu Fujino for various suggestions, constant support, and warm encouragement. The author is also grateful to Professor Kenta Hashizume for explaining the key arguments noted in the text and valuable comments on the earlier draft. The author also would like to thank Professor Makoto Enokizono for helpful discussions. The author is indebted to Hiromu Tanaka for sharing his expertise on related topics.
\end{ack}

\section{Preliminaries}
\begin{defn}[Gorenstein, $\mathbb{Q}$-Gorenstein, or $\mathbb{Q}$-factorial surfaces]
Let $X$ be a normal projective surface.
We say that $X$ is \textit{Gorenstein} (resp.~ \textit{$\mathbb{Q}$-Gorenstein}) if the canonical divisor $K_X$ is Cartier (resp.~ $\mathbb{Q}$-Cartier).
We say that $X$ is \textit{$\mathbb{Q}$-factorial} if every prime Weil divisor on $X$ is $\mathbb{Q}$-Cartier.
\end{defn}
\begin{rem}[Numerical properties of Weil divisors in the sense of Mumford]
Following \cite{Sak84}, we briefly recall that for a Weil divisor on a normal surface $X$, its pull-back to a resolution of singularities is well-defined. This construction allows us to define the intersection number of Weil divisors on $X$. Consequently, numerical properties such as being nef can be extended to Weil divisors; to avoid confusion with the standard definitions for Cartier divisors, these are sometimes referred to as being \textit{nef in the sense of Mumford.} Since the projection formula remains valid in this generalized setting (cf.~ \cite[Theorem (2.1)]{Sak84}), the study of the section ring $R(X, D)$ for a Weil divisor $D$ can be reduced to that of $R(Y, D_Y)$ by taking a log resolution $f\colon(Y, D_Y) \to (X, D)$.
\end{rem}
\begin{defn}[{Canonical rings and Kodaira dimension (cf.~ \cite{Sak87})}]
Let $X$ be a normal projective surface.
The \textit{canonical ring} $R(X)$ of $X$ is defined as
$$
R(X):=\bigoplus_{m\geq0}H^0(X, \mathcal{O}_X(mK_X)).
$$
We also define the \textit{Kodaira dimension} of $X$ (as a normal surface) by 
$$
\kappa(X):=\mathrm{tr.deg.}_k R(X)-1
$$ 
but $\kappa(X)=-\infty$ if $R(X)=k$ (cf.~ \cite[Definition (2.3)]{Sak84}). 
This definition is distinct from the original one by Iitaka, in which the Kodaira dimension of singular variety is defined via a smooth model (see, for example, \cite[Examples 2.1.5 and 2.1.6]{Laz04}).
\end{defn}
\begin{defn}
Let $D=\sum_ia_iD_i$ be a divisor on a variety $X$, where $D_i$'s are distinct prime divisors. The \textit{round-down} of $D$ is denoted by
$$
\llcorner D \lrcorner:=\sum_i\llcorner a_i\lrcorner D_i.
$$
\end{defn}

\section{A criterion of finite generation}\label{sec: criterion}
In this section, we extend a criterion for the finite generation, originally given for $k = \mathbb{C}$  (cf.~ \cite[Theorem 4.7]{Sak87}), to arbitrary characteristic.

First, we establish a version of the basepoint-free theorem for non-klt pairs in arbitrary characteristic. This result follows from the basepoint-free theorem for klt pairs (cf.~ \cite[Theorem 3.3]{KM98} and \cite[Theorem 3.2]{Tan15}). Such a reduction was originally suggested by Kawamata (cf.~ \cite[Lemma 3]{Kaw92}).

\begin{thm}\label{thm: basepoint-free}
Let $X$ be a projective normal surface over $k$ (in arbitrary characteristic) and let $\Delta$ be an effective $\mathbb{Q}$-divisor such that $K_X + \Delta$ is $\mathbb{Q}$-Cartier. 
Let $D$ be a nef Cartier divisor.
If $\mathrm{char}(k)>0$, we further assume that $D$ is not numerically trivial. 
Suppose that $aD - (K_X + \Delta)$ is nef and big for some $a \in \mathbb{Z}_{>0}$, and that $D|_{\mathrm{Nklt}(X, \Delta)}$ is semi-ample. Then $D$ is semi-ample.
\end{thm}
The following argument is a slight modification of \cite[Theorem 3.2]{Tan15}.
\begin{proof}
Consider the following exact sequence:
\begin{align*}
0\rightarrow \mathcal{J}_\Delta\otimes\mathcal{O}_X(mD)
\rightarrow \mathcal{O}_X(mD)
\rightarrow \mathcal{O}_{\mathrm{Nklt}(X, \Delta)}(mD)
\rightarrow 0,
\end{align*}
where $\mathcal{J}_\Delta$ denotes the multiplier ideal sheaf of $(X, \Delta)$.
The $H^1$ of the first term vanishes for all $m\gg 0$.
This vanishing follows from \cite[Theorem 2.9]{Tan15} in positive characteristic. In characteristic zero, it follows from the case $k=\mathbb{C}$ (\cite[Theorem 3.4.2]{Fuj17}) via the Lefschetz principle, as argued in the proof of \cite[Theorem 2.4]{Fuj21-MR}.
Since $D|_{\mathrm{Nklt}(X, \Delta)}$ is semi-ample, this implies that
$$
\mathrm{Bs}(mD)\cap\mathrm{Nklt}(X, \Delta)=\emptyset
$$
for all $m\gg 0$.
The subsequent argument is identical to that in \cite[Theorem 3.2]{Tan15}; hence, it is omitted here. We note that, while \cite{Tan15} only treats the case of positive characteristic, the argument remains valid in characteristic zero.
This is because the required vanishing theorem is well-known over $\mathbb{C}$, and it can be extended to general characteristic zero via the Lefschetz principle.
\end{proof}

The following result is well-known. For the original treatment in arbitrary characteristic, see \cite{Zar62}. Modern references include \cite[Remark 2.3]{Wil81} and \cite[Theorem 2.3.15]{Laz04} for the characteristic zero case. It follows from Fujita's vanishing theorem, which holds in arbitrary characteristic \cite[Section 3.8]{Fuj17}. Although the proof in \cite{Laz04} assumes $k = \mathbb{C}$, the same argument works for any algebraically closed field.
\begin{thm}\label{thm: nef and big}
Let $D$ be a nef and big Cartier divisor on a normal projective variety $X$. Then the section ring $R(X, D)$ is finitely generated if and only if $D$ is semi-ample.
\end{thm}

We now provide a criterion that was originally shown in \cite{Sak87} for $k =\mathbb{C}$. Since the proof of \cite[Theorem 4.7]{Sak87} remains valid in any characteristic, the same result holds in the positive characteristic setting. For the reader's convenience, we include the proof below.

\begin{thm}[{cf. \cite[Theorem 4.7]{Sak87}}]\label{thm: criterion}
Let $k$ be an algebraically closed field of any characteristic.
Let $X$ be a normal projective surface over $k$ such that $\kappa(X)=2$ and $K_X$ is nef in the sense of Mumford. Then, $R(X)$ is finitely generated if and only if $X$ is $\mathbb{Q}$-Gorenstein.  
\end{thm}
\begin{proof}
\textbf{The \textit{if} part.}
Let $K_X$ be $\mathbb{Q}$-Cartier.
Let $\pi \colon Y \to X$ be the minimal resolution, and let $K_Y + \Delta = \pi^* K_X$.
Then, $K_Y+\Delta$ is nef and big and $\mathrm{Supp}(\mathrm{Nklt}(Y, \Delta))$ is contained in the exceptional locus of $\pi$. We find that $(K_Y+\Delta)|_{\mathrm{Nklt}(Y, \Delta)}$ is trivial.
So, we can show that $K_Y+\Delta$ is semi-ample by Theorem \ref{thm: basepoint-free}.
Then, $R(Y, K_Y+\Delta)\cong R(X)$ is finitely generated.

\noindent
\textbf{The \textit{only if} part.}
Let $(Y, \pi, \Delta)$ be the minimal resolution of $X$, and $A = \bigcup E_i$ the exceptional set. Note that
$$
R(X) \cong R(Y, K_Y + \Delta):=\bigoplus_{m \ge 0} H^0(Y, \mathcal{O}_Y(\llcorner m(K_Y+\Delta))\lrcorner).
$$
is finitely generated.
Since $K_Y + \Delta$ is a nef and big $\mathbb{Q}$-Cartier $\mathbb{Q}$-divisor, this implies that $K_Y + \Delta$ is semi-ample by Theorem \ref{thm: nef and big}. Then, its some multiple is free.
Write $\mathcal{L} = \mathcal{O}_Y(m(K_Y + \Delta))$ for such $m$. Since $\mathcal{L}\cdot E_i=0$ for all $i$, $\mathcal{L}$ is numerically trivial near $A$. This means that $\pi_*\mathcal{L}=\mathcal{O}_X(mK_X)$ near $\pi(A)$. Since $A$ is the exceptional set of $\pi$,  we have $\pi_*\mathcal{L}=\mathcal{O}_X(mK_X)$ and $\pi$ is the morphism induced by the free line bundle $\mathcal{L}$. Hence $\mathcal{O}_X(mK_X)$  is invertible.
\end{proof}

\section{Counterexamples to the MMP} \label{sec: MMP}
In this section, we recall Sakai's example \cite[Example 4.5]{Sak87}, 
which was treated under the condition $k=\mathbb{C}$, and we will check that the construction of this example is valid in positive characteristic.
This example implies that $\mathbb{Q}$-Gorensteinness is not preserved under the MMP. In contrast, Gorensteinness and $\mathbb{Q}$-factoriality are preserved (see Lemma \ref{lem: MMP for Gorenstein} and \cite{Fuj21-MR}, respectively).

The following lemma was suggested to the author by Enokizono, Fujino, and Hashizume.
\begin{lem}[The MMP for Gorenstein surfaces]\label{lem: MMP for Gorenstein}
The MMP holds for normal projective Gorenstein surfaces in any characteristic.
More precisely, for any normal projective Gorenstein surface $X$ over $k$ (of arbitrary characteristic), there is a sequence of contractions
$$
X=X_0\stackrel{\phi_0}{\rightarrow}X_1\stackrel{\phi_1}{\rightarrow}\cdots\stackrel{\phi_{k-1}}{\rightarrow}X_k=X^*
$$
such that either $K_{X^*}$ is a nef Cartier divisor or $X^*$ is a Mori fiber space.
\end{lem}
A part of the argument in the proof is parallel to the one in  \cite[Theorem 4.10.8]{Fuj17}.
\begin{proof}
We consider the following step-by-step procedure.
\begin{itemize}
\item[Step 1.] We start with $X$.
\item[Step 2.] If $K_X$ is nef, we stop.
\item[Step 3.] If $K_X$ is not nef, then there is a rational $K_X$-negative extremal curve $C$ and we have the corresponding contraction morphism $\varphi_C\colon X\rightarrow X'$ by \cite[Theorems 5.5 and 5.6]{Fuj21-MR}. 
We note that $X'$ is normal.
\item[Step 4.] If $\varphi_C\colon X\rightarrow X'$ is not birational, then it is a Mori fiber space and we stop.
\item[Step 5.] Otherwise, $\varphi_C\colon X\rightarrow X'$ is birational.
Then, we will replace $X$ with $X'$ and go back to the first procedure.
\end{itemize}
To complete the proof, it suffices to show that $X'$ is Gorenstein in the last step.
Let $\varphi_C\colon X\rightarrow X'$ be the birational contraction associated to a rational $K_X$-negative extremal curve $C$. Then we find that $C^2<0$ in the sense of Mumford.
Let $f\colon Y\rightarrow X$ be the minimal resolution. We can write $K_Y=f^*K_X+\sum_i a_iE_i$, where $E_i$'s are distinct $f$-exceptional divisors. By the negativity lemma, $a_i\leq0$ for each $a_i$. Since $0> K_X\cdot C=(K_Y-\sum_i a_iE_i)\cdot f_*^{-1}C$, it follows that $K_Y\cdot f_*^{-1}C<0$. This implies $K_Y\cdot f_*^{-1}C=-1$ and $(f_*^{-1}C)^2=-1$. Consequently, $(\sum_i a_iE_i)\cdot f_*^{-1}C=0$, meaning $X$ has canonical singularity around $C$. In particular, $C$ is a $\mathbb{Q}$-Cartier divisor on $X$ and $\mathrm{Supp(Exc}(\varphi_C))=C$. Thus, $X'$ has canonical singularities on a neighborhood of the point $\varphi_C(C)$. Then, the point $\varphi_C(C)$ is either a smooth point or a Du Val singularity; hence, $K_{X'}$ is Cartier on a neighborhood of $\varphi_C(C)$ (see, for example, \cite[Theorem 4.5]{KM98} and \cite[Theorems 3.15, 3.31 and Corollary 4.19]{Bad01}). 
Since $\varphi_C$ is an isomorphism outside $C$, $K_{X'}$ is Cartier on the whole of $X'$, making $X'$ Gorenstein. Finally, since the Picard number $\rho(X)$ decreases by exactly one at each birational step, where $\rho(X)$ denotes the dimension of the $\mathbb{R}$-vector space $N^1(X):=\{\mathrm{Pic}(X)/\equiv\}\otimes\mathbb{R}$, this procedure must terminate.
\end{proof}

In Sakai's example, the contraction is constructed using the Grauert--Sakai contraction criterion, and the projectivity of the resulting surface is shown by Brenton's projectivity criteria. However, since Brenton's criteria are results in the analytic setting, they are not applicable
when $k\neq\mathbb{C}$.

On the other hand, Sakai's example can be reconstructed in a more modern framework. 
The strategy of using the basepoint-free theorem to construct a contraction morphism via a semi-ample line bundle was suggested to the author by Hashizume. While the original suggestion was in the context of $k=\mathbb{C}$, we have here generalized and adapted the argument to work in arbitrary characteristic by verifying Theorem \ref{thm: basepoint-free}.
Alternatively, while one could apply Keel's result (cf.~ \cite[Theorem 2.2]{Tan14}) in the case of positive characteristic, we shall not do so here to provide a uniform treatment across all characteristics.

\vspace{1em} 
\noindent 
\textbf{Construction.}
Assume that either $\mathrm{char}(k) = 0$ 
or $\mathrm{char}(k) > 0$ with $k \neq \bar{\mathbb{F}}_p$. 
Let $C$ be an elliptic curve over $k$, and choose a non-torsion element 
$\mathfrak{e} \in \mathrm{Pic}^0(C)$; the existence of such an element is 
guaranteed by our assumption on $k$.
Let
$$
S := \mathbb{P}_C(\mathcal{O}_C \oplus \mathcal{O}_C(\mathfrak{e})) \xrightarrow{p} C
$$
be the structure of a geometrically ruled surface. Then there are two sections $B$ and $B'$ of $p$ such that $B' \sim B - p^*\mathfrak{e}$. Note that $K_S + B + B' \sim 0$ and $B^2 = B'^2 = B \cdot B' = 0$. Fix two distinct fibers $F$ and $F'$ of $p$. 

We first take a blow-up $X_1$ of $S$ at $B \cap F$ and let $E_1$ be the exceptional divisor.

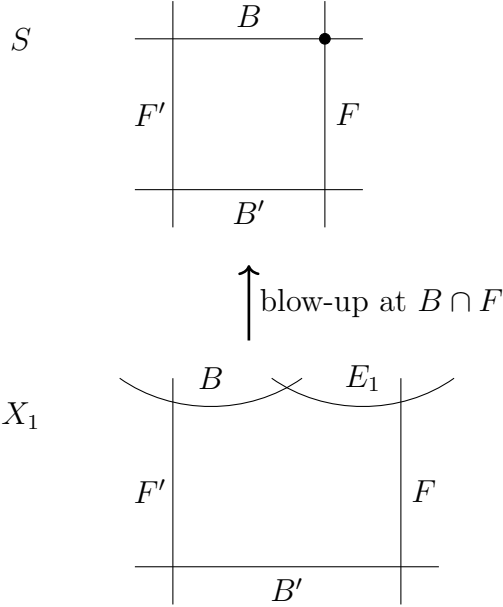
\begin{tikzpicture}
\draw (-1, -1) node{$S$};
\draw (0.5, -1) -- (3.5, -1);
\draw (2, -0.7) node{$B$};
\draw (0.5, -3) -- (3.5, -3);
\draw (2, -3.3) node{$B'$};
\draw (1, -0.5) -- (1, -3.5);
\draw (0.7, -2) node{$F'$};
\draw (3, -0.5) -- (3, -3.5);
\draw (3.3, -2) node{$F$};
\filldraw (3, -1) circle (2pt);

\draw[->, line width=1pt] (2, -5) -- (2, -4);
\draw (2, -4.5) node[anchor=west]{blow-up at $B\cap F$};

\draw (-1, -6) node{$X_1$};
\draw (0.3, -5.5) .. controls (1, -6) and (2, -6) .. (2.7, -5.5);
\draw (1.5, -5.5) node{$B$};
\draw (2.3, -5.5) .. controls (3, -6) and (4, -6) .. (4.7, -5.5);
\draw (3.5, -5.5) node{$E_1$};

\draw (0.5, -8) -- (4.5, -8);
\draw (2.5, -8.3) node{$B'$};
\draw (1, -5.5) -- (1, -8.5);
\draw (0.7, -7) node{$F'$};
\draw (4, -5.5) -- (4, -8.5);
\draw (4.3, -7) node{$F$};
\end{tikzpicture}

The pull-back of $K_S$ is $K_{X_1}-E_1$, the pull-back of $B$ is $B+E_1$, and the pull-back of $B'$ is $B'$, where we denote the strict transform of a prime divisor $D$ simply by $D$, by abuse of notation.

We take a blow-up $X:=X_2$ of $X_1$ at $E_1\cap F$ and let $E_2$ be the exceptional divisor.

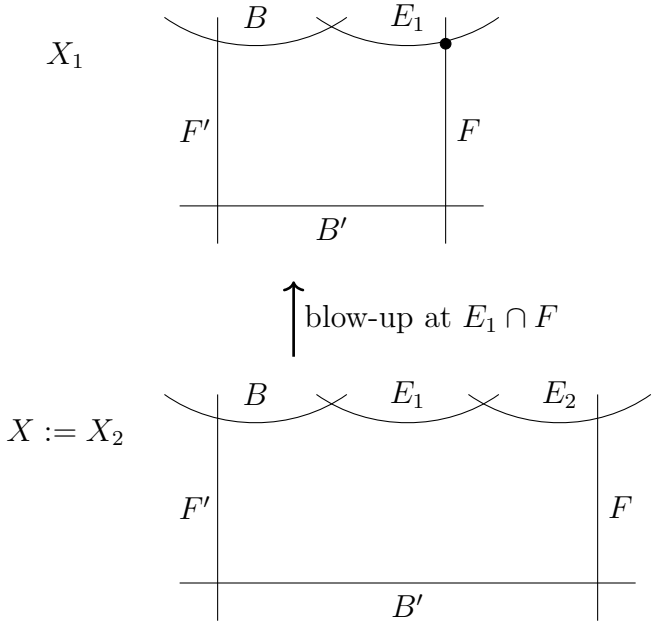
\begin{tikzpicture}
\draw (-1, -1) node{$X_1$};
\draw (0.3, -0.5) .. controls (1, -1) and (2, -1) .. (2.7, -0.5);
\draw (1.5, -0.5) node{$B$};
\draw (2.3, -0.5) .. controls (3, -1) and (4, -1) .. (4.7, -0.5);
\draw (3.5, -0.5) node{$E_1$};

\draw (0.5, -3) -- (4.5, -3);
\draw (2.5, -3.3) node{$B'$};
\draw (1, -0.5) -- (1, -3.5);
\draw (0.7, -2) node{$F'$};
\draw (4, -0.5) -- (4, -3.5);
\draw (4.3, -2) node{$F$};
\filldraw (4, -0.85) circle (2pt);

\draw[->, line width=1pt] (2, -5) -- (2, -4);
\draw (2, -4.5) node[anchor=west]{blow-up at $E_1\cap F$};

\draw (-1, -6) node{$X:=X_2$};
\draw (0.3, -5.5) .. controls (1, -6) and (2, -6) .. (2.7, -5.5);
\draw (1.5, -5.5) node{$B$};
\draw (2.3, -5.5) .. controls (3, -6) and (4, -6) .. (4.7, -5.5);
\draw (3.5, -5.5) node{$E_1$};
\draw (4.3, -5.5) .. controls (5, -6) and (6, -6) .. (6.7, -5.5);
\draw (5.5, -5.5) node{$E_2$};

\draw (0.5, -8) -- (6.5, -8);
\draw (3.5, -8.3) node{$B'$};
\draw (1, -5.5) -- (1, -8.5);
\draw (0.7, -7) node{$F'$};
\draw (6, -5.5) -- (6, -8.5);
\draw (6.3, -7) node{$F$};
\end{tikzpicture}

Then we get a projective birational morphism $\varphi: X \to S$. 
We have $\varphi^*K_S=K_X-E_1-2E_2$, $\varphi^*B=B+E_1+E_2$, $\varphi^*B'=B'$, and $\varphi^*F=F+E_1+2E_2$.
Since $K_S+B+B'\sim 0$, we have $K_X+B+B'-E_2\sim 0$.

\begin{lem} \label{lem: B and E_1}
The divisor $3B'+F$ on $X$ is semi-ample and we can take a contraction morphism $\pi\colon X\rightarrow Y$. The contraction $\pi$ contracts only the curves $B$ and $E_1$.
\end{lem}
\begin{proof}
We note that $B'^2=0$ and $F^2=-2$.
The divisor $3B'+F$ on $X$ is nef because $(3B'+F)\cdot B'=1$ and $(3B'+F)\cdot F=1$.
Then, $3B'+F$ is big because $(3B'+F)^2=4$.
We also note that for any irreducible curve $C_0$ on $X$, we have $(3B'+F)\cdot C_0=0$ if and only if $C_0$ is either $B$ or $E_1$. 
This follows from the following calculations:
\begin{align*}
3B' + F 
&\sim B' + 2\varphi^*(B - p^*\mathfrak{e})+F \\
&\sim B' + 2(B + E_1 + E_2) - 2\varphi^*p^*\mathfrak{e} +F \\
&\sim (\varphi^*(B' + F) - 2\varphi^*p^*\mathfrak{e}) + 2B + E_1,
\end{align*}
where $B'+F$ is an ample divisor on $S$, and we have
\begin{align*}
(3B' + F )\cdot B=0, \\
(3B' + F )\cdot E_1=0, \\
(3B' + F )\cdot E_2=1.
\end{align*}

The divisor $B+F$ on $S$ is ample and $p^*\mathfrak{e}$ is numerically trivial.
We have
\begin{align*}
    &\varphi^*(B+F)=B+F+2E_1+3E_2,\\
    &\varphi^*p^*\mathfrak{e}=\varphi^*(B-B')=B-B'+E_1+E_2.
\end{align*}

Since $K_X+B+B'-E_2\sim 0$, we have
\begin{align*}
3B'+F-K_X 
&\sim 3B'+F+B+B'-E_2\\
&\sim 4(B+E_1+E_2-\varphi^*p^*\mathfrak{e})+F+B-E_2\\
&\sim 5B+4E_1+3E_2+F-4\varphi^*p^*\mathfrak{e}\\
&\sim 5B+4E_1+3E_2+(\varphi^*(B+F)-B-2E_1-3E_2)-4\varphi^*p^*\mathfrak{e}\\
&\sim 4B+2E_1+\varphi^*(B+F-4p^*\mathfrak{e}).
\end{align*}

Since $\mathfrak{e} \in \mathrm{Pic}^0(B)$ and $B + F$ is ample, we see that $\varphi^*(B+F-4p^*\mathfrak{e})$ is nef and big.
By construction, we can easily check that 
$\mathrm{Nklt}(X, 4B + 2E_1) = \llcorner 4B + 2E_1 \lrcorner$.
Since $B'\cup F$ is disjoint from $B\cup E_1$, 
we have $(3B'+F)|_{\mathrm{Nklt}(X, 4B + 2E_1)}\sim 0$. In particular, $(3B'+F)|_{\mathrm{Nklt}(X, 4B + 2E_1)}$ is semi-ample.
By applying Theorem \ref{thm: basepoint-free} to $3B'+F$, 
we see that $3B'+F$ is semi-ample.
Then we can take a contraction morphism $\pi\colon X\rightarrow Y$ and $\pi$ contracts only the curves $B$ and $E_1$.
\end{proof}

\begin{lem}
The given surface $Y$ is not $\mathbb{Q}$-Gorenstein.
\end{lem}
\begin{proof}
We may write
$$
K_X=\pi^*K_Y+aB+bE_1,
$$
where $\pi^*K_Y$ is the pull-back in the sense of Mumford.
Suppose to the contrary that $Y$ is $\mathbb{Q}$-Gorenstein.
Then we have $(K_X-aB-bE_1)|_B\sim_\mathbb{Q} 0$.
Since $K_X+B+B'-E_2\sim 0$ and $\varphi^*p^*\mathfrak{e}\sim\varphi^*(B-B')\sim B-B'+E_1+E_2$,
we have
\begin{align*}
0\sim_\mathbb{Q} (\pi^*K_Y)|_B &\sim_\mathbb{Q} (K_X-aB-bE_1)|_B\\
    &\sim (-B-B'+E_2-aB-bE_1)|_B\\
    &= ((B-B'+E_1+E_2)-(a+2)B-(b+1)E_1)|_B\\
    &\sim (\varphi^*p^*\mathfrak{e}-(a+2)B-(b+1)E_1)|_B.
\end{align*}
On the other hand, since $B\cong C$, we have
$$
0 \sim_\mathbb{Q} \mathfrak{e}+(-(a+2)B-(b+1)E_1)|_B.
$$
Since $B^2=-1$, $B\cdot E_1=1$ and $E_1^2=-2$, we have $a-b+1=0$.

We also have
\begin{align*}
0\sim_\mathbb{Q} (\pi^*K_Y)|_{E_1}
&\sim_\mathbb{Q} (K_X-aB-bE_1)|_{E_1}\\
    &\sim (-B-B'+E_2-aB-bE_1)|_{E_1}\\
    &= ((B-B'+E_1+E_2)-(a+2)B-(b+1)E_1)|_{E_1}\\
    &\sim (\varphi^*p^*\mathfrak{e}-(a+2)B-(b+1)E_1)|_{E_1}.
\end{align*}
Thus, $-a+2b=0$.
We then have $a=-2$ and $b=-1$.
This means that
$$
0 \sim_\mathbb{Q} \mathfrak{e}+(-(a+2)B-(b+1)E_1)|_B =\mathfrak{e}.
$$
This contradicts the fact that $\mathfrak{e}$ is non-torsion.
Therefore, $K_Y$ is not $\mathbb{Q}$-Cartier. We also have
\begin{align}\label{a=-2, b=-1}
\pi^*K_Y=K_X+2B+E_1,
\end{align}
where $\pi^*K_Y$ is the pull-back in the sense of Mumford.
\end{proof}

We take a blow-up $X_3$ of $X$ at $B \cap F'$ and let $E_3$ be the exceptional divisor.

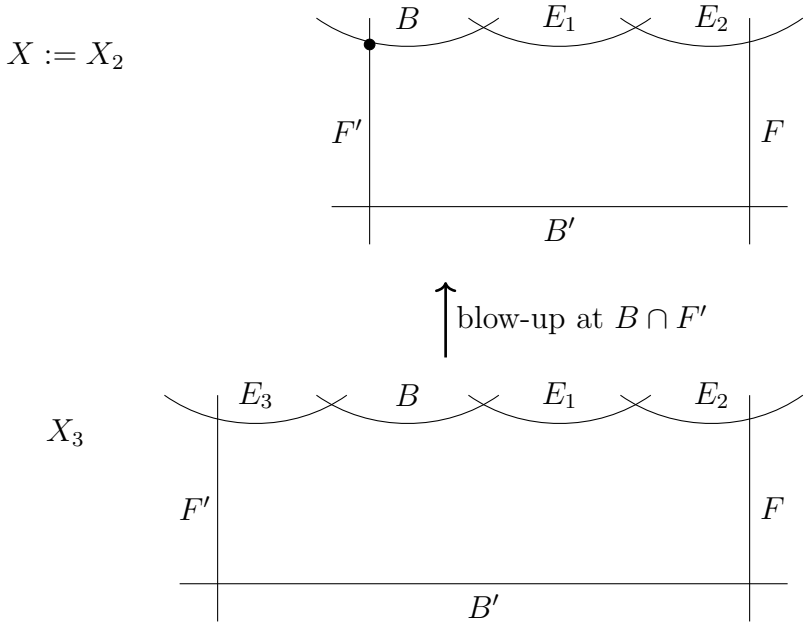
\begin{tikzpicture}
\draw (-3, -1) node{$X:=X_2$};
\draw (0.3, -0.5) .. controls (1, -1) and (2, -1) .. (2.7, -0.5);
\draw (1.5, -0.5) node{$B$};
\draw (2.3, -0.5) .. controls (3, -1) and (4, -1) .. (4.7, -0.5);
\draw (3.5, -0.5) node{$E_1$};
\draw (4.3, -0.5) .. controls (5, -1) and (6, -1) .. (6.7, -0.5);
\draw (5.5, -0.5) node{$E_2$};

\draw (0.5, -3) -- (6.5, -3);
\draw (3.5, -3.3) node{$B'$};
\draw (1, -0.5) -- (1, -3.5);
\draw (0.7, -2) node{$F'$};
\draw (6, -0.5) -- (6, -3.5);
\draw (6.3, -2) node{$F$};
\filldraw (1, -0.85) circle (2pt);

\draw[->, line width=1pt] (2, -5) -- (2, -4);
\draw (2, -4.5) node[anchor=west]{blow-up at $B\cap F'$};

\draw (-3, -6) node{$X_3$};
\draw (-1.7, -5.5) .. controls (-1, -6) and (0, -6) .. (0.7, -5.5);
\draw (-0.5, -5.5) node{$E_3$};
\draw (0.3, -5.5) .. controls (1, -6) and (2, -6) .. (2.7, -5.5);
\draw (1.5, -5.5) node{$B$};
\draw (2.3, -5.5) .. controls (3, -6) and (4, -6) .. (4.7, -5.5);
\draw (3.5, -5.5) node{$E_1$};
\draw (4.3, -5.5) .. controls (5, -6) and (6, -6) .. (6.7, -5.5);
\draw (5.5, -5.5) node{$E_2$};

\draw (-1.5, -8) -- (6.5, -8);
\draw (2.5, -8.3) node{$B'$};
\draw (-1, -5.5) -- (-1, -8.5);
\draw (-1.3, -7) node{$F'$};
\draw (6, -5.5) -- (6, -8.5);
\draw (6.3, -7) node{$F$};
\end{tikzpicture}

The pull-back of $K_S$ is $K_{X_3}-E_1-2E_2-E_3$, the pull-back of $B$ is $B+E_1+E_2+E_3$, and the pull-back of $B'$ is $B'$.

We take a blow-up $X_4$ of $X_3$ at $E_3\cap F'$ and let $E_4$ be the exceptional divisor.

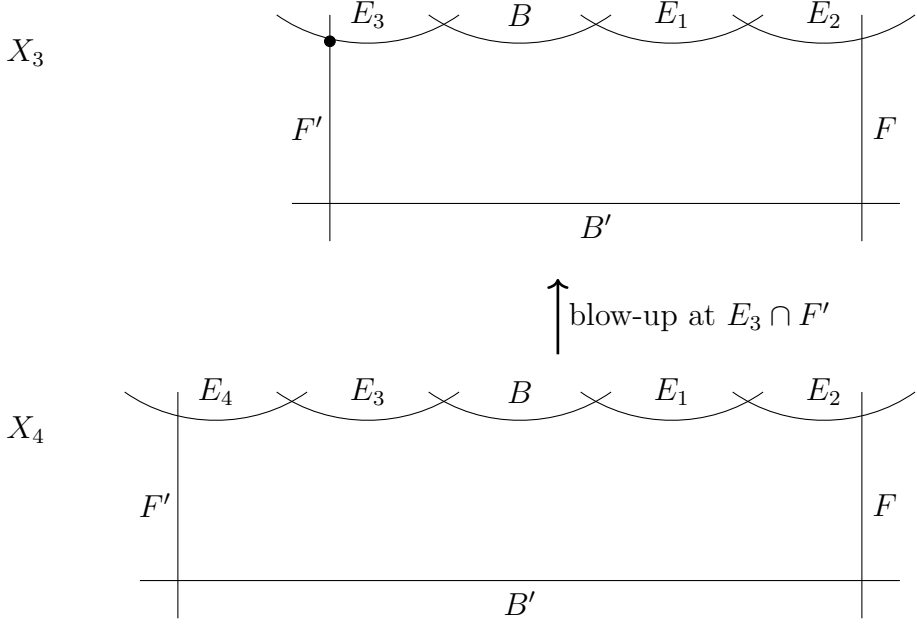
\begin{tikzpicture}
\draw (-5, -1) node{$X_3$};
\draw (-1.7, -0.5) .. controls (-1, -1) and (0, -1) .. (0.7, -0.5);
\draw (-0.5, -0.5) node{$E_3$};
\draw (0.3, -0.5) .. controls (1, -1) and (2, -1) .. (2.7, -0.5);
\draw (1.5, -0.5) node{$B$};
\draw (2.3, -0.5) .. controls (3, -1) and (4, -1) .. (4.7, -0.5);
\draw (3.5, -0.5) node{$E_1$};
\draw (4.3, -0.5) .. controls (5, -1) and (6, -1) .. (6.7, -0.5);
\draw (5.5, -0.5) node{$E_2$};

\draw (-1.5, -3) -- (6.5, -3);
\draw (2.5, -3.3) node{$B'$};
\draw (-1, -0.5) -- (-1, -3.5);
\draw (-1.3, -2) node{$F'$};
\draw (6, -0.5) -- (6, -3.5);
\draw (6.3, -2) node{$F$};
\filldraw (-1, -0.85) circle (2pt);

\draw[->, line width=1pt] (2, -5) -- (2, -4);
\draw (2, -4.5) node[anchor=west]{blow-up at $E_3\cap F'$};

\draw (-5, -6) node{$X_4$};
\draw (-3.7, -5.5) .. controls (-3, -6) and (-2, -6) .. (-1.3, -5.5);
\draw (-2.5, -5.5) node{$E_4$};
\draw (-1.7, -5.5) .. controls (-1, -6) and (0, -6) .. (0.7, -5.5);
\draw (-0.5, -5.5) node{$E_3$};
\draw (0.3, -5.5) .. controls (1, -6) and (2, -6) .. (2.7, -5.5);
\draw (1.5, -5.5) node{$B$};
\draw (2.3, -5.5) .. controls (3, -6) and (4, -6) .. (4.7, -5.5);
\draw (3.5, -5.5) node{$E_1$};
\draw (4.3, -5.5) .. controls (5, -6) and (6, -6) .. (6.7, -5.5);
\draw (5.5, -5.5) node{$E_2$};

\draw (-3.5, -8) -- (6.5, -8);
\draw (1.5, -8.3) node{$B'$};
\draw (-3, -5.5) -- (-3, -8.5);
\draw (-3.3, -7) node{$F'$};
\draw (6, -5.5) -- (6, -8.5);
\draw (6.3, -7) node{$F$};
\end{tikzpicture}

The pull-back of $K_S$ is $K_{X_4}-E_1-2E_2-E_3-2E_4$, the pull-back of $B$ is $B+E_1+E_2+E_3+E_4$, and the pull-back of $B'$ is $B'$.

We take a blow-up $\tilde{X}:=X_5$ of $X_4$ at $E_4\cap F'$ and let $E_5$ be the exceptional divisor.

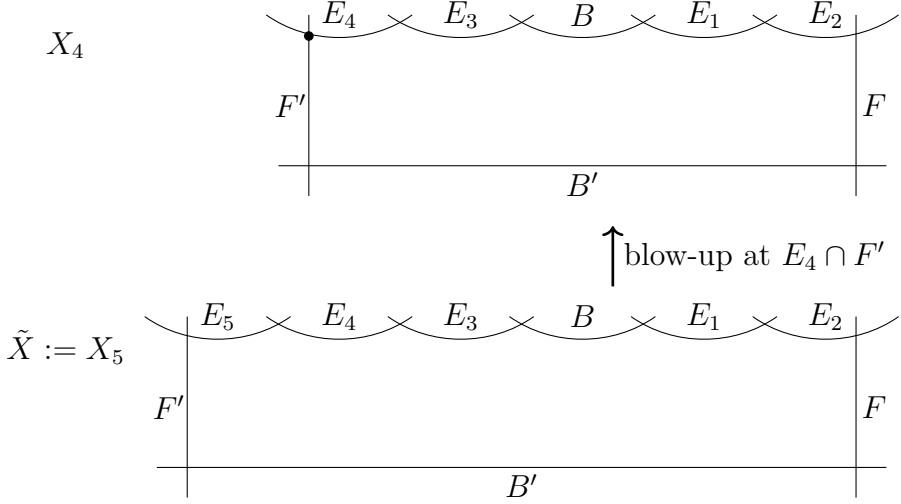
\begin{tikzpicture}[scale=0.8]
\draw (-7, -1) node{$X_4$};
\draw (-3.7, -0.5) .. controls (-3, -1) and (-2, -1) .. (-1.3, -0.5);
\draw (-2.5, -0.5) node{$E_4$};
\draw (-1.7, -0.5) .. controls (-1, -1) and (0, -1) .. (0.7, -0.5);
\draw (-0.5, -0.5) node{$E_3$};
\draw (0.3, -0.5) .. controls (1, -1) and (2, -1) .. (2.7, -0.5);
\draw (1.5, -0.5) node{$B$};
\draw (2.3, -0.5) .. controls (3, -1) and (4, -1) .. (4.7, -0.5);
\draw (3.5, -0.5) node{$E_1$};
\draw (4.3, -0.5) .. controls (5, -1) and (6, -1) .. (6.7, -0.5);
\draw (5.5, -0.5) node{$E_2$};

\draw (-3.5, -3) -- (6.5, -3);
\draw (1.5, -3.3) node{$B'$};
\draw (-3, -0.5) -- (-3, -3.5);
\draw (-3.3, -2) node{$F'$};
\draw (6, -0.5) -- (6, -3.5);
\draw (6.3, -2) node{$F$};
\filldraw (-3, -0.85) circle (2pt);

\draw[->, line width=1pt] (2, -5) -- (2, -4);
\draw (2, -4.5) node[anchor=west]{blow-up at $E_4\cap F'$};

\draw (-7, -6) node{$\tilde{X}:=X_5$};
\draw (-5.7, -5.5) .. controls (-5, -6) and (-4, -6) .. (-3.3, -5.5);
\draw (-4.5, -5.5) node{$E_5$};
\draw (-3.7, -5.5) .. controls (-3, -6) and (-2, -6) .. (-1.3, -5.5);
\draw (-2.5, -5.5) node{$E_4$};
\draw (-1.7, -5.5) .. controls (-1, -6) and (0, -6) .. (0.7, -5.5);
\draw (-0.5, -5.5) node{$E_3$};
\draw (0.3, -5.5) .. controls (1, -6) and (2, -6) .. (2.7, -5.5);
\draw (1.5, -5.5) node{$B$};
\draw (2.3, -5.5) .. controls (3, -6) and (4, -6) .. (4.7, -5.5);
\draw (3.5, -5.5) node{$E_1$};
\draw (4.3, -5.5) .. controls (5, -6) and (6, -6) .. (6.7, -5.5);
\draw (5.5, -5.5) node{$E_2$};

\draw (-5.5, -8) -- (6.5, -8);
\draw (0.5, -8.3) node{$B'$};
\draw (-5, -5.5) -- (-5, -8.5);
\draw (-5.3, -7) node{$F'$};
\draw (6, -5.5) -- (6, -8.5);
\draw (6.3, -7) node{$F$};
\end{tikzpicture}

Then we get a projective birational morphisms $\tilde{\varphi}: \tilde{X} \to S$ and $\mu\colon\tilde{X}\rightarrow X$.
We have $\tilde{\varphi}^*K_S=K_{\tilde{X}}-E_1-2E_2-E_3-2E_4-3E_5$, $\tilde{\varphi}^*B=B+E_1+E_2+E_3+E_4+E_5$, and $\tilde{\varphi}^*B'=B'$.
Since $K_S+B+B'\sim 0$, we have $K_{\tilde{X}}+B+B'-E_2-E_4-2E_5\sim 0$.

\begin{lem}
The divisor $ 6B'+2F+F'$ on $\tilde{X}$ is semi-ample and we can take a contraction morphism $\tilde{\pi}\colon \tilde{X}\rightarrow \tilde{Y}$. The contraction $\tilde{\pi}$ contracts only the curves $B$, $E_1$, $E_3$, and $E_4$.
\end{lem}
\begin{proof}
We note that $B'^2=0$, $F^2=-2$ and $F'^2=-3$.
The divisor $6B'+2F+F'$ on $\tilde{X}$ is nef because $(6B'+2F+F')\cdot B'=3$, $(6B'+2F+F')\cdot F=2$ and $(6B'+2F+F')\cdot F'=3$.
Then, $6B'+2F+F'$ is big because $(6B'+2F+F')^2=25$.
We also note that for any irreducible curve $C_0$ on $X$, we have $(6B'+2F+F')\cdot C_0=0$ if and only if $C_0$ is either $B$, $E_1$, $E_3$ or $E_4$. 
This follows from the following calculations:
\begin{align*}
6B'+2F+F'=2\mu^*(3B'+F)+F',
\end{align*}
which implies by the proof of Lemma \ref{lem: B and E_1} that candidates of the $\tilde{\pi}$-exceptional prime divisors are $B$, $E_1$, $E_3$, $E_4$, $E_5$, and $F'$. We also have
\begin{align*}
(6B'+2F+F')\cdot B=0,\\
(6B'+2F+F')\cdot E_1=0,\\
(6B'+2F+F')\cdot E_3=0,\\
(6B'+2F+F')\cdot E_4=0,\\
(6B'+2F+F')\cdot E_5=1,\\
(6B'+2F+F')\cdot F'=3
\end{align*}

The divisor $2B+2F+F'$ on $S$ is ample and $p^*\mathfrak{e}$ is numerically trivial.
We have
\begin{align*}
    &\tilde{\varphi}^*(2B+2F+F')\\
    &=2(B+E_1+E_2+E_3+E_4+E_5)+2(F+E_1+2E_2)+(F'+E_3+2E_4+3E_5)\\
    &=2B+4E_1+6E_2+3E_3+4E_4+5E_5+2F+F',\\
    &\tilde{\varphi}^*p^*\mathfrak{e}=\tilde{\varphi}^*(B-B')=B-B'+E_1+E_2+E_3+E_4+E_5.
\end{align*}
Since $K_{\tilde{X}}+B+B'-E_2-E_4-2E_5\sim 0$, we have
\begin{align*}
&6B'+2F+F'-K_{\tilde{X}}\\
&\sim 6B'+2F+F'+B+B'-E_2-E_4-2E_5\\
&\sim 7(B+E_1+E_2+E_3+E_4+E_5-\tilde{\varphi}^*p^*\mathfrak{e})+2F+F'+B-E_2-E_4-2E_5\\
&\sim 8B+7E_1+6E_2+7E_3+6E_4+5E_5+2F+F'-7\tilde{\varphi}^*p^*\mathfrak{e}\\
&\sim 8B+7E_1+6E_2+7E_3+6E_4+5E_5\\
&+(\tilde{\varphi}^*(2B+2F+F')-(2B+4E_1+6E_2+3E_3+4E_4+5E_5))-7\tilde{\varphi}^*p^*\mathfrak{e}\\
&\sim 6B+3E_1+4E_3+2E_4+\tilde{\varphi}^*(2B+2F+F'-7p^*\mathfrak{e}).
\end{align*}
Since $\mathfrak{e} \in \mathrm{Pic}^0(B)$ and $2B+2F+F'$ is ample, we see that $\tilde{\varphi}^*(2B+2F+F'-7p^*\mathfrak{e})$ is nef and big.
By construction, we can easily check that 
$\mathrm{Nklt}(\tilde{X}, 6B+3E_1+4E_3+2E_4) = \llcorner 6B+3E_1+4E_3+2E_4 \lrcorner$.
Since $B'\cup F\cup F'$ is disjoint from $B\cup E_1\cup E_3\cup E_4$, we have
$(6B'+2F+F')|_{\mathrm{Nklt}(\tilde{X}, 6B+3E_1+4E_3+2E_4)}\sim 0$. In particular, $(6B'+2F+F')|_{\mathrm{Nklt}(\tilde{X}, 6B+3E_1+4E_3+2E_4)}$ is semi-ample.
By applying Theorem \ref{thm: basepoint-free} to $6B'+2F+F'$,
we see that $6B'+2F+F'$ is semi-ample.
Then we can take a contraction morphism $\tilde{\pi}\colon \tilde{X}\rightarrow \tilde{Y}$ and $\tilde{\pi}$ contracts only the curves $B$, $E_1$, $E_3$ and $E_4$.
\end{proof}

Although the choice of the two distinct fibers $F$ and $F'$ was arbitrary, we will show that $\tilde{Y}$ can be made $\mathbb{Q}$-Gorenstein by choosing them appropriately.
\begin{lem}
The fibers $F$ and $F'$ on $S$ can be chosen such that the resulting surface $\tilde{Y}$ is $\mathbb{Q}$-Gorenstein.
\end{lem}
\begin{proof}
We first calculate the discrepancies of the exceptional divisors $B$, $E_1$, $E_3$ and $E_4$ for the morphism $\tilde{\pi}\colon\tilde{X}\rightarrow\tilde{Y}$.
We write
$$
K_{\tilde{X}}=\tilde{\pi}^*K_{\tilde{Y}}+aB+bE_1+cE_3+dE_4,
$$
where $\tilde{\pi}^*K_{\tilde{Y}}$ is the pull-back in the sense of Mumford.
By calculation, we find that $B^2=E_1^2=E_3^2=E_4^2=-2$ for these four divisors on $\tilde{X}$.
We have
\begin{align*}
0=(\tilde{\pi}^*K_{\tilde{Y}})\cdot B
&= (K_{\tilde{X}}-aB-bE_1-cE_3-dE_4) \cdot B\\
&=(-(a+1)B-B'-bE_1+E_2-cE_3-(d-1)E_4+2E_5)\cdot B.
\end{align*}
Then,
\begin{equation}\label{eq: 1}
0=2(a+1)-b-c. 
\end{equation}
We have
\begin{align*}
0= (-(a+1)B-B'-bE_1+E_2-cE_3-(d-1)E_4+2E_5)\cdot {E_1}.
\end{align*}
Then,
\begin{equation}\label{eq: 2}
0=-(a+1)+2b+1. 
\end{equation}
We have
\begin{align*}
0= (-(a+1)B-B'-bE_1+E_2-cE_3-(d-1)E_4+2E_5)\cdot {E_3}.
\end{align*}
Then,
\begin{equation}\label{eq: 3}
0=-(a+1)+2c-(d-1). 
\end{equation}
We have
\begin{align*}
0= (-(a+1)B-B'-bE_1+E_2-cE_3-(d-1)E_4+2E_5)\cdot {E_4}.
\end{align*}
Then,
\begin{equation}\label{eq: 4}
0=-c+2(d-1)+2. 
\end{equation}
By equations \ref{eq: 1}, \ref{eq: 2}, \ref{eq: 3} and \ref{eq: 4}, we get
$a=-\frac{12}{5}$, $b=-\frac{6}{5}$, $c=-\frac{8}{5}$ and $d=-\frac{4}{5}$.
Now, we have
$$
K_{\tilde{X}}=\tilde{\pi}^*K_{\tilde{Y}}-\frac{1}{5}(12B+6E_1+8E_3+4E_4),
$$
where $\tilde{\pi}^*K_{\tilde{Y}}$ is the pull-back in the sense of Mumford.
Let $F\cap B=:x\in S$ and $F'\cap B=:x'\in S$ be the points on the curve $B$ corresponding to the fibers $F$ and $F'$, respectively.
We have $\tilde{\varphi}^*F =F + E_1 + 2E_2$, $\tilde{\varphi}^*F' = F' + E_3 + 2E_4 + 3E_5$. 
We regard $\mathfrak{a} = x - x'$ as a divisor on $C$ by identifying $B$ with $C$. Then
\[
\tilde{\varphi}^*(p^*\mathfrak{a}) \sim E_1 + 2E_2 - E_3 - 2E_4 - 3E_5 \quad \text{near } B\cup E_1\cup E_3\cup E_4.
\]
Since $B'$ is disjoint from $B$, $\mathcal{O}_X(B')$ is trivial near $B$ and so 
$$
\tilde{\varphi}^*B \sim \tilde{\varphi}^*(p^*\mathfrak{e})
\quad \text{near } B\cup E_1\cup E_3\cup E_4.
$$
On the other hand,
\[
K_{\tilde{X}} \sim -\tilde{\varphi}^*(p^*\mathfrak{e}) + E_1 + 2E_2 + E_3 + 2E_4 + 3E_5
 \quad \text{near } B\cup E_1\cup E_3\cup E_4.
\]
Combining these together, we have
\begin{align*}
5(K_{\tilde{X}} + \frac{1}{5}(12B + 6E_1 + 8E_3+ 4E_4))
&\sim 5(-\tilde{\varphi}^*(p^*\mathfrak{e}) + E_1 + 2E_2 + E_3 + 2E_4 + 3E_5)\\
&+12(\tilde{\varphi}^*(p^*\mathfrak{e})-E_1-E_2-E_3-E_4-E_5)\\
&+6E_1 + 8E_3+ 4E_4\\
&=7\tilde{\varphi}^*(p^*\mathfrak{e})-E_1-2E_2+E_3+2E_4+3E_5\\
&\sim \tilde{\varphi}^*(p^*(7\mathfrak{e} - \mathfrak{a}))
\end{align*}
near $B\cup E_1\cup E_3\cup E_4$.
We deduce from this that $\tilde{Y}$ is $\mathbf{Q}$-Gorenstein if and only if $7\mathfrak{e} - \mathfrak{a}$ is a torsion element in $\mathrm{Pic}^0(B)$, where $\mathfrak{a} \sim x - x'$.
This condition holds if
\[
x - x' \sim_{\mathbb{Q}} 7\mathfrak{e} \quad \text{in} \quad \text{Pic}^0(B).
\]
Since we can choose the positions of the points $x$ and $x'$ on $B$ arbitrarily, there exist fibers $F$ and $F'$ satisfying this condition.
\end{proof}

Finally, we will demonstrate that this example shows that $\mathbb{Q}$-Gorensteinness is not necessarily preserved under the MMP.
The composition $\tilde{X}\rightarrow X\rightarrow Y$ contracts $B$, $E_1$, $E_3$, $E_4$ and $E_5$ and $\tilde{X}\rightarrow\tilde{Y}$ contracts $B$, $E_1$, $E_3$ and $E_4$. Then, there exists a contraction $\tilde{Y}\rightarrow Y$ which contracts $E_5$.

\begin{tikzpicture}[scale=0.8]
\draw[->] (3, -2.5) -- (8, -2.5);
\draw[->] (3, -4.5) -- (8, -4.5);

\draw (2.5, -2.5) node{$\tilde{X}$};
\draw[->] (2.5, -3) -- (2.5, -4);
\draw (2.5, -3.5) node[anchor=east]{$\tilde{\pi}$};

\draw (8.5, -2.5) node{$X$};
\draw[->] (8.5, -3) -- (8.5, -4);
\draw (8.5, -3.5) node[anchor=west]{$\pi$};

\draw (2.5, -4.5) node{$\tilde{Y}$};
\draw (0.3, -5.5) .. controls (1, -6) and (2, -6) .. (2.7, -5.5);
\draw (1.5, -5.5) node{$E_5$};
\draw (2.3, -5.5) .. controls (3, -6) and (4, -6) .. (4.7, -5.5);
\draw (3.5, -5.5) node{$E_2$};
\draw (0.5, -8) -- (4.5, -8);
\draw (2.5, -8.3) node{$B'$};
\draw (1, -5.5) -- (1, -8.5);
\draw (0.7, -7) node{$F'$};
\draw (4, -5.5) -- (4, -8.5);
\draw (4.3, -7) node{$F$};

\draw[->] (5, -7) -- (6, -7);

\draw (8.5, -4.5) node{$Y$};

\draw (6.3, -5.5) .. controls (7, -6) and (10, -6) .. (10.7, -5.5);
\draw (8.5, -5.5) node{$E_2$};
\draw (6.5, -8) -- (10.5, -8);
\draw (8.5, -8.3) node{$B'$};
\draw (7, -5.5) -- (7, -8.5);
\draw (6.7, -7) node{$F'$};
\draw (10, -5.5) -- (10, -8.5);
\draw (10.3, -7) node{$F$};
\end{tikzpicture}

In particular, we have $E_5^2<0$ for this divisor on $\tilde{Y}$.
We also have
\begin{align*}
K_{\tilde{Y}}\cdot E_5
&=(\tilde{\pi}^*K_{\tilde{Y}})\cdot E_5\\
&=(-B-B'+E_2+E_4+2E_5+\frac{1}{5}(12B+6E_1+8E_3+4E_4))\cdot E_5\\
&=\frac{9}{5}-2<0.
\end{align*}
Then, the contraction $\tilde{Y}\rightarrow Y$ is given by a $K_{\tilde{Y}}$-negative extremal ray of the Kleiman--Mori cone $\overline{\textnormal{NE}}(\tilde{Y})$.
This implies that $\mathbb{Q}$-Gorensteinness is not necessarily preserved in the process of going to the minimal model (cf.~ \cite[Remark 4.4]{Sak87}).

\section{Counterexamples to the finite generation} \label{sec: f.g.}
In this section, we construct an example of a normal projective $\mathbb{Q}$-Gorenstein surface with $\kappa=2$ whose canonical ring is not finitely generated. 
It is worth noting that the finite generation holds for Gorenstein surface (cf.~ Remark \ref{rem: f.g. for Gorenstein}), and for $\mathbb{Q}$-factorial surface (cf.~ \cite{Fuj12} and \cite{Tan14}).
The following proof of Lemma \ref{lem: kappa=1} and the subsequent construction were suggested to the author by Fujino.

\begin{lem}\label{lem: kappa=1}
Let $X$ be a normal projective surface over $k$ (of arbitrary characteristic) with $\kappa(X)\leq 1$. Then, $R(X)$ is finitely generated.
\end{lem}
\begin{proof}
Take a resolution $f\colon (X', \Delta)\rightarrow X$. Then we have $R(X)=R(X', K_{X'}+\Delta)$.
Let $\phi\colon \tilde{X}\rightarrow Y$ be the Iitaka fibration associated to $\mathcal{O}_{X'}(K_{X'}+\Delta)$ and let $\nu\colon\tilde{X}\rightarrow X'$ be the induced birational morphism. We can write 
$K_{\tilde{X}}+\tilde{\Delta}=\nu^*(K_{X'}+\Delta)$. We note that $Y$ is a smooth curve and $\phi$ is flat. By \cite[(1.12) Theorem, (i)]{Mor87}, there exsists $c\in \mathbb{Z}_{>0}$ such that $\phi_*\mathcal{O}_{\tilde{X}}(c(K_{\tilde{X}}+\tilde{\Delta}))\cong\mathcal{O}_Y(M)$ is an invertible sheaf and the natural map $\mathcal{O}_Y(mM)\rightarrow \phi_*\mathcal{O}_{\tilde{X}}(mc(K_{\tilde{X}}+\tilde{\Delta}))$ is an isomorphism for any $m\in\mathbb{Z}_{>0}$. Since $\kappa(Y, M)=\kappa(\tilde{X}, K_{\tilde{X}}+\tilde{\Delta})=1$, the divisor $M$ is ample. In particular, $R(Y, M)$ is finitely generated. This implies that $R(X)=R(X', K_{X'}+\Delta)=R(\tilde{X}, K_{\tilde{X}}+\tilde{\Delta})$ is finitely generated.
\end{proof}

\begin{rem}[The finite generation for Gorenstein surfaces]\label{rem: f.g. for Gorenstein}
For Gorenstein surfaces in any characteristic, the canonical ring is finitely generated (cf.~ \cite[Theorem 14.42]{Bad01} or \cite[Theorem 4.10.9]{Fuj17}). We include a short proof for the reader's convenience.
Let $X$ be a normal projective Gorenstein surface over $k$ (of arbitrary characteristic).
By Lemma \ref{lem: MMP for Gorenstein}, we may assume that $K_X$ is nef.
If $\kappa(X)\leq1$, it follows from Lemma \ref{lem: kappa=1}.
If $\kappa(X)=2$, it follows from Theorem \ref{thm: criterion}.
\end{rem}

\vspace{1em}
\noindent 
\textbf{Construction.}
We take the contraction $\tilde{Y}\rightarrow Y$ of the example in Section \ref{sec: MMP}.
We note that $K_Y$ is nef in the sense of Mumford because
\begin{align*}
\pi^*K_Y &= K_X+2B+E_1 \mathrm{\ (by\ the\ equation\ }\eqref{a=-2, b=-1})\\
    &\sim B-B'+E_1+E_2\\
    &\sim \varphi^*p^*\mathfrak{e}.
\end{align*}
Since $Y$ is projective, we can take a very ample divisor $H$ on $Y$. 
Take a general member $A\in|mH|$. Since $H$ is very ample, we may assume that $\mathrm{Supp}A$ is contained in the smooth locus of $X$. We consider a degree $m$ cyclic cover 
$$
g\colon Z=\text{Spec}_Y(\bigoplus_{l=0}^{m-1}\mathcal{O}_Y(-lH))\rightarrow Y
$$
ramified along $A$ (cf.~ \cite[Definition 2.50]{KM98}). Here, we set $m = 2$ in characteristic zero; in positive characteristic, we assume that $m$ is coprime to $\mathrm{char}(k)$.
We also take the fiber product $\tilde{Z}:=\tilde{Y}\times_Y Z$ and let $\tilde{g}\colon \tilde{Z}\rightarrow \tilde{Y}$ be the induced finite cover.
\begin{center}
\begin{tikzcd}
\widetilde{Z} := \widetilde{Y} \times_Y Z \arrow[r] \arrow[d, "\tilde{g}" '] & Z \arrow[d, "g"] \\
\widetilde{Y} \arrow[r] & Y
\end{tikzcd}   
\end{center}
Then we have
$$
K_Z\sim g^*(K_Y+\frac{m-1}{m}A)
$$
by the ramification formula.
Since $g$ is finite and $Y$ is not $\mathbb{Q}$-Gorenstein, $Z$ is not $\mathbb{Q}$-Gorenstein. On the other hand, we find that $\kappa(Z)=2$. Then $K_Z$ is nef in the sense of Mumford and big and the canonical ring $R(Z)$ is not finitely generated by Theorem \ref{thm: criterion}.
Then, the canonical ring $R(\tilde{Z})$ is not finitely generated.
We also note that $\kappa(\tilde{Z})=\kappa(Z)=2$ and $\tilde{Z}$ is $\mathbb{Q}$-Gorenstein because $\tilde{Y}$ is $\mathbb{Q}$-Gorenstein.
Therefore, this $\tilde{Z}$ is an example of normal projective $\mathbb{Q}$-Gorenstein surface with $\kappa(\tilde{Z)}=2$ whose canonical ring is not finitely generated.

\bibliographystyle{amsalpha} 
\bibliography{ref}

\end{document}